\documentclass{amsart}

\usepackage{graphicx}

\newtheorem{theorem}{Theorem}

\newtheorem{lem}{Lemma}

\newtheorem{remark}{Remark}

\newtheorem{problem}{Problem}

\newtheorem{example}{Example}

\newtheorem{THEO}{Theorem}

\newcommand{\la}{\lambda}

\newcommand{\bC}{\mathbb C}

\newcommand{\bR}{\mathbb R}

\newcommand {\dq}{\mathfrak d}

\newcommand {\Log}{\text{\,Log\,}}

\newcommand {\grad}{\text{\,grad\,}}

\begin{document}

\title[Electrostatic problems with a constraint and Lam\'e equations]{Electrostatic problems with a rational constraint and degenerate Lam\'e equations}

\author{Dimitar K. Dimitrov}

\address{Departamento de Matem\'atica Aplicada, IBILCE, Universidade Estadual Paulista\\
15054-000 S\~{a}o Jos\'{e} do Rio Preto, SP, Brazil}
\email{dimitrov@ibilce.unesp.br}

\thanks{Research supported by the
Brazilian Science Foundations FAPESP under Grants 2016/09906-0 and 2017/02061-8 and CNPq under Grant 306136/2017-1.}



\author {Boris Shapiro}
\address{Department of Mathematics, Stockholm University, SE-106 91
Stockholm,
        Sweden}
\email{shapiro@math.su.se}

\dedicatory {\quad \quad \quad \quad To Heinrich Eduard Heine and Thomas Joannes Stieltjes
 \newline whose mathematics continues  to inspire after more than a century}




\begin{abstract} 

In this note we extend the classical relation between the  equilibrium configurations of unit movable point charges in a plane electrostatic field created by these charges together with some fixed point charges
 and the polynomial solutions of a corresponding  Lam\'e differential equation. Namely,   we find similar relation between the equilibrium configurations of unit movable charges  subject to a certain type of rational or polynomial constraint and  polynomial solutions of a corresponding degenerate Lam\'e equation, see details below. 
 In particular,  the standard linear differential equations satisfied by the classical Hermite and Laguerre 
polynomials  belong to this class.  Besides these two classical cases, we present a number of other examples including  some relativistic 
orthogonal polynomials and  linear differential equations satisfied by those. 
\end{abstract}
 
\maketitle

\begin{center}
\section{Introduction}
\end{center}



For a given configuration of $p+1$ fixed point charges $\nu_{j}$ located at the fixed points $a_{j}\in \bC$ 
 and  $n$  unit movable charges  
 located at the variable points $x_k\in \bC,\; k=1,\dots, n$ respectively, 
 the (logarithmic) energy of this configuration  is  given by
\begin{equation}\label{L} L(x_{1}, \ldots, x_{n}) =  - \sum_{k=1}^{n} \sum_{j=0}^{p} \nu_j  \log
{|a_j -x_{k}|}\  - \sum_{1\leq i <k \leq n}  \log
{|x_{k}-x_{i}|}.
\end{equation}



\medskip 
The standard electrostatic problem in this set-up is as follows. 

\medskip
\begin{problem}
\label{pr:basic} 
 Find/count all equilibrium configurations  of movable charges, i.e.,  all the critical points of the energy function  $L(x_{1},\ldots,x_{n})$. 
 \end{problem}



\medskip
The above  electrostatic problem has been initially studied  by H.~E.~Heine \cite{He} and T.~J.~Stieltjes \cite{StB} and is now commonly known as 
the  {\it classical Heine-Stieltjes electrostatic problem}. Besides Heine and Stieltjes, the latter question has been considered by  F.~Klein \cite{Kl}, E.~B.~Van Vleck \cite{VVl},  G.~Szeg\H{o} \cite{Sz},  and G.~P\'olya \cite{Pol}, just to mention a few. A relatively recent survey of results on the classical Heine-Stieltjes problems can be found in \cite{Boris11_1}. 
The general case of arbitrary movable charges is much less studied due to the missing relation  with linear ordinary differential equations, but the case when all movable charges are of the same sign was, in particular, considered by A.~Varchenko in \cite{Va}.

\begin{THEO}[Stieltjes' theorem, \cite{StB}]\label{th:StB} 
If all $p+1$ fixed positive charges are placed on the real line, then for each of the 
$(n+p-1)!/(n!(p-1)!)$ possible placements of  $n$ unit movable charges in  $p$ finite intervals of the real axis bounded by the fixed charges,  
  the classical Heine-Stieltjes problem possesses a unique solution. 
  \end{THEO}
  
  Heine's major result from \cite{He} claims that in the situation with $p+1$ fixed and $n$ unit movable charges, the number of  equilibrium configurations (assumed finite) can not exceed $(n+p-1)!/(n!(p-1)!)$, see Theorem~\ref{th:He}   below. Therefore, the above Stieltjes' theorem describes all  possible equilibrium configurations occurring under  the assumptions of Theorem~\ref{th:StB}.  
  
  \medskip
The most essential observation of the Heine-Stieltjes theory is that  in case of equal movable charges,   each equilibrium configuration is described by  a polynomial solution of a {\it Lam\'e differential 
equation}. Recall that   a Lam\'e equation (in its algebraic form) is given by 
\begin{equation}
A(x) y^{\prime\prime} + 2 B(x) y^{\prime} + V(x) y = 0,
\label{LameeqN}
\end{equation}
where $A(x)=(x-a_{0})\cdots(x-a_{p})$ and $B(x)$ is a 
polynomial of degree  at most $p$ such that 
\begin{equation}
\frac{B(x)}{A(x)} = \sum_{j=0}^{p} \frac{\nu_j}{x-a_j}
\label{ParfracN}
\end{equation}
and $V(x)$ is a polynomial of degree at most $p-1$.

\medskip 
In this set-up the general Heine-Stieltjes electrostatic problem is equivalent to the  following question about the corresponding  Lam\'e equation:

\medskip
\begin{problem}\label{pr:HS} Given polynomials $A(x)$ and $B(x)$ as above, and a positive integer $n$, find all possible polynomials $V(x)$ of degree at most $p-1$ for which 
equation~\eqref{LameeqN} has a polynomial solution $y$ of degree $n$. 
\end{problem}

Heine's original result was formulated in this language and it claims the following. 

\begin{THEO} [Heine \cite{He}, see also \cite{Boris11_1}]\label{th:He} If the coefficients of the polynomials $A(x)$ and $B(x)$ are algebraically independent numbers, i.e., they do not satisfy an algebraic equation with integer coefficients, 
then for any integer $n > 0,$ there exist exactly $\binom{n+p-1}{n}$  polynomials $V(x)$ of  degree $p-1$ 
 such that the equation \eqref{LameeqN} has a unique (up to a constant factor) polynomial solution $y$ of degree  $n$.
\end{THEO}

Polynomial $V(x)$ solving Problem~\ref{pr:HS} is called a {\it Van Vleck polynomial} of the latter problem while the corresponding polynomial solution $y(x)$ is called the {\it Stieltjes polynomial} corresponding to $V(x)$. 

\medskip
The relation between Problem~\ref{pr:basic} and Problem~\ref{pr:HS} in case of equal movable charges is very straightforward. Namely, given the locations and values $(a_j,\nu_j),\; j=0,\dots, p$ of the fixed charges and the number $n$ of the movable unit charges, every equilibrium configuration of the movable charges is exactly the set of all zeros of some Stieltjes polynomial $y(x)$ of degree $n$  for Problem~\ref{pr:HS}. 

\medskip
In the  particular case of  $p=1$, $a_0=-1$, $a_1=1$, $\nu_0=(\beta+1)/2$ and 
$\nu_1=(\alpha+1)/2,$  this interpretation explains why  the unique equilibrium position of $n$ unit movable charges in the interval $(-1,1)$ is determined by the fact that the set 
$\{x_k\}$, $k=1,\ldots,n$ coincides with the zero locus of the Jacobi polynomial $P_{n}^{(\alpha,\beta)}(x)$. (Recall that  $\{P_{n}^{(\alpha,\beta)}(x)\}$ is the sequence of polynomials, orthogonal on $[-1,1]$ with respect to the weight function $(1-x)^{\alpha}(1+x)^{\beta}$.) 




 
\medskip
The main  goal of this note is to find the electrostatic interpretation of the zeros of polynomial solutions of  Problem~\ref{pr:HS}  for a more general class of Lam\'e equations when there is no  restriction $\deg A(x) > \deg B(x)$. 
 Namely, 
we say that a linear second order differential operator 
\begin{equation}\label{Lame}
\dq=A(x)\frac{d^2}{dx^2}+2B(x)\frac{d}{dx},
\end{equation}
with polynomial coefficients $A(x)$ and $B(x)$ is a \emph{non-degenerate Lam\'e operator} if $\deg A(x)>\deg B(x)$ and a \emph{degenerate Lam\'e operator} otherwise. The \emph{Fuchs index} $f_\dq$ of the operator $\dq$ is, by definition, given by 
$$f_\dq:=max(\deg A(x)-2, \deg B(x)-1).$$
For the Heine-Stieltjes problem 
to be well-defined, we consider below Lam\'e ope\-ra\-tors $T$ with $f_\dq\ge 0$.  For such operators, Problem~\ref{pr:HS} with equation~\eqref{LameeqN} makes perfect sense. In other words,   we  are looking for Van Vleck polynomials $V(x)$ of degree at most $f_\dq$ such that \eqref{LameeqN}  has a polynomial solution of a given degree $n$. We call equation \eqref{LameeqN} with a degenerate Lam\'e operator $\dq=A(x)\frac{d^2}{dx^2}+B(x)\frac{d}{dx}$ and $V(x)$ of degree at most $f_\dq$, a \emph{degenerate Lam\'e equation}.

\medskip
The most well-known examples of such equations are those satisfied by the Hermite  and the Laguerre polynomials. Namely,     the Hermite and the Laguerre   polynomials are polynomial solutions of the  second order 
differential equations:

\begin{equation}\label{HDE}
y^{\prime\prime} - 2 x\,  y^{\prime} + 2n\,  y = 0,\ \ \ y(x)=H_n(x),
\end{equation}

\begin{equation}\label{LDE}
x\, y^{\prime\prime} + (\alpha+1-x)\, y^{\prime} + n\,  y = 0,\ \ \ y(x)=L_n^{\alpha}(x).
\end{equation}

\medskip Obviously, \eqref{HDE} and \eqref{LDE} are degenerate Lam\'e equations with Fuchs index $0$. 
Due to this fact,  the classical interpretation of the zeros of the Hermite and the Laguerre polynomials as  coordinates of  the critical points of the energy function \eqref{L} does not apply. Nevertheless, as was already observed  by G.~Szeg\H{o} \cite{Sz},    the zeros of the Hermite  and the Laguerre polynomials still possess a nice electrostatic interpretation   in terms of the minimum of the  energy function \eqref{L}  subject to certain  polynomial constraints.

\begin{THEO}[\cite{Sz}, Theorems 6.7.2 and 6.7.3]\label{HSZ}

{\rm (i)} The zeros of the Hermite polynomial $H_n(c_2 x)$, $c_2=\sqrt{(n-1)/2M},\; M>0$, form an equilibrium configuration of
 $n$ unit charges that obey the constraint $x_1^2 + \cdots + x_n^2  \leq M n$ without  fixed charges. In particular, the zeros of $H_n(x)$  form an equilibrium configuration of
 $n$ unit charges subject to the constraint 
$x_1^2 + \cdots + x_n^2 \leq n(n-1)/2.$ 

\noindent
{\rm (ii)} The zeros of the Laguerre polynomial $L_n^{\alpha}(c_1x)$, $c_1=(n+\alpha)/K,\; K>0$, form an equilibrium configuration of
 $n$ unit movable charges  that obey the additional constraint $x_1 + \cdots + x_n  \leq K n$ in the presence of one fixed charge $\nu_0=(\alpha+1)/2$ placed at the origin. In particular, the zeros of $L_n^{(\alpha)}(x)$  form an equilibrium configuration of $n$ unit charges satisfying the constraint 
$
   x_1 + \cdots + x_n \leq n (n+\alpha)
 $ 
in the  electrostatic field created  by them together with the above fixed charge. 

 \end{THEO}
 
\begin{remark} {\rm Due to homogeneity of the problem, one can easily conclude  that it suffices to consider only the case of equality in the above constraints. }     
\end{remark}

From the first glance it seems difficult to relate the  electrostatic problems of Theorem~\ref{HSZ} to the corresponding differential equation since e.g., in the case 
of the Hermite polynomials
$$
B(x)/A(x) = -x
$$
which has no non-trivial partial fraction decomposition.  


\medskip
In order to do this, we restrict ourselves to degenerate Lam\'e equations with all distinct roots of $A(x)$ and non-negative Fuchs index. Given such an equation, consider  the classical electrostatic problem with  all unit movable charges and assume that 
 the positions  $X=(x_1,\ldots,x_n)$ of the movable 
charges are subject to an additional constraint $R(X)=0$ of a special form.
Namely, we say that a rational function $R(X)$ is \emph {$A$-adjusted} if 
\begin{equation}\label{eq:AA}
R(X):=R(x_1,\dots, x_n)= \sum_{k=1}^n {r}(x_k),
\end{equation}
 where $r(x)$ is a fixed univariate rational 
function such that $D(x):=A(x) r^\prime(x)$ is a polynomial.  We will also call the univariate function $r(x)$ satisfying the latter condition \emph {$A$-adjusted}. In particular, if $r(x)$ is an arbitrary polynomial, then \eqref{eq:AA} is automatically $A$-adjusted. 

\medskip   Given an $A$-adjusted rational function $R(X)$, 
set $$\Omega = \{ X\in \bC^n\ :\ R(X)=0 \}$$ and denote by  $\mathcal A$  the hyperplane arrangement in $\bC^n$ with coordinates $(x_1,\dots,x_n)$ consisting of all hyperplanes of the form  $\{x_i=a_j\}$ and $\{x_i=x_\ell\}$ for $i=1,\dots, n$, $j=0, \dots, p$ and $i\neq \ell$. Here $\{a_0,\dots, a_p\}$ is the set of roots of $A(x)$ (assumed pairwise distinct). 

\medskip
Now consider the $1$-parameter family of the Lam\'e differential equations  of the form 
\begin{equation}
\label{MLDE}
A(x) y^{\prime\prime} + (2  B(x) - \rho D(x)) y^{\prime} + V(x) y
= 0,
\end{equation}
depending on a complex-valued parameter $\rho$. We call \eqref{MLDE} the \emph{parametric Lam\'e equation}.

\medskip

Consider an arbitrary degenerate Lam\'e equation \eqref{LameeqN} with all distinct roots of $A(x)$. 
 Given an $A$-adjusted rational function $R(X)$, set $q:=\deg D(x).$

\begin{theorem}[Stieltjes' theorem 
 with an $A$-adjusted constraint] 
\label{ThR}
 In the above notation,  the  energy function $L(x)$ given by \eqref{L} and the parametric Lam\'e equation \eqref{MLDE} satisfy the following:

\medskip
\begin{itemize}
\item Let $X^\ast=(x_1^\ast, \dots, x_n^\ast)$ be a vector  lying in $\bC^n\setminus \mathcal A$.  Assume that $X^\ast$ satisfies the constraint $R(X^\ast)=0$ and is  a critical  point of the energy function $L(X)$. Then there exist a polynomial $V(x)$ of degree  $\max\{p-1,q-1\}$ and a  
constant $\rho^\ast$ such that $$y(x)= (x-x_1^\ast)\cdots (x-x_n^\ast)$$ is a
solution of the parametric Lam\'e equation {\rm (\ref{MLDE})} with $\rho=\rho^\ast$.

\medskip
\item Let $V(x)$ be a polynomial satisfying the condition $deg\, V \leq \max \{ p-1,q-1\}$, and $\rho^\ast$
be a constant for which the parametric Lam\'e equation {\rm (\ref{MLDE})} possesses a
polynomial solution of the form $y(x)=(x-x_1^\ast)\cdots(x-x_n^\ast)$, such that $X^\ast=(x_1^\ast, \dots, x_n^\ast) \in \Omega \setminus \mathcal A$. Then
$$
\left. \frac{\partial L(X)}{\partial x_k} \right|_{X^\ast} - \, \rho^\ast \left.  \frac{\partial
R(X)}{\partial x_k} \right|_{X^\ast} = 0,\ \ k=1,\ldots, n.
$$
\end{itemize}
Here, by definition,   $\left. \frac{\partial R(X)}{\partial x_k}\right|_{X^\ast}:=r^\prime(x_k^\ast)$.  
 \end{theorem}
 
 \medskip
\begin{remark}
{\rm  Additionally, if all fixed charges are positive and placed on the real line, and $X^\ast$ is a point of a local minimum of $L(X)$, then the constant $\rho^\ast$ must be positive. }
 \end{remark}
 
 \begin{example}{\rm 
 Observe  that the above differential equations for the Hermite and Laguerre polynomials are nothing else but the parametric Lam\'e 
 equations for the appropriate values  of parameter $\rho$ and they adequately describe  the corresponding electrostatic problems.  (The corresponding values of $\rho$ are denoted by $\rho^\ast$.)
 
\medskip 
Namely,  the Hermite polynomial $H_n(x)$ is a solution of the  differential equation (\ref{HDE}) which can be interpreted as a   parametric Lam\'e equation \eqref{MLDE} with  
$A(x)=1$, $B(x)=0$, $V(x)=2n$, $D(x)=2x$, and $\rho^\ast=1/2$. The fact that $A(x)=1$ and $B(x)=0$ is equivalent to the absence  
of fixed charges. In this case,    $R(X)=x_1^2+ \cdots + x_n^2 -(n-1)/2n$ which implies that 
$$
D(x_k) := A(x_k) \frac{\partial R(X)}{\partial x_k} =A(x_k)r^\prime(x_k)= 2 x_k,
$$
where $r(x)=x^2-(n-1)/2n^2$. 

Analogously, the differential equation (\ref{LDE}) satisfied by the Laguerre polynomial $L_n^{(\alpha)}(x)$ can be interpreted as a 
parametric Lam\'e equation  \eqref{MLDE} with $A(x)=x$, $B(x) = (\alpha+1)/2$, $D(x)=1$, and $\rho^\ast=1/2$. Then the partial fraction decomposition
$B(x)/A(x)= (\alpha+1)/2x$ indicates the presence of one fixed charge $(\alpha+1)/2$ at the origin.   In this case, 
$R(X)=x_1 + \cdots + x_n -(n+\alpha)/n$ and} 
$$
D(x_k) := A(x_k) \frac{\partial R(X)}{\partial x_k} =A(x_k)r^\prime(x_k)= 1, 
$$
where $r(x)=x- (n+\alpha)/n^2$.
\end{example}

\medskip
Theorem~\ref{ThR} allows us to formulate the main result of this note which provides  a general  relation between  degenerate Lam\'e equations and electrostatic problems  
in the presence of an $A$-adjusted rational constraint.   
\begin{theorem}
\label{th:Deg}
Let  
\begin{equation}
\label{DLDE}
A(x) y^{\prime\prime} + 2{B}(x)  y^{\prime} + V(x) y
= 0,
\end{equation}
be a degenerate Lam\'e equation, i.e., $\deg A(x)\leq \deg {B}(x)$ and $\deg V \le f_\dq$ for $\dq= A(x) y^{\prime\prime} + 2 {B}(x)  y^{\prime}$. 
Assume that all roots of $A(x)$ are distinct and that $r(x)$ is an $A$-adjusted univariate rational function  such that $\deg (B(x)-A(x) r^\prime(x))<\deg A(x)$.   
Set 
$$
B (x):=A(x) r^\prime(x) + \tilde B(x), \quad\quad 
R(X)=R(x_1,\dots, x_n) := \rho + \sum_{k=1}^n r(x_k),$$
where $\rho$ is an arbitrary  complex constant.  Then  there exists a  value $\rho^\ast$ of the constant $\rho$,  for which  the degenerate Lam\'e equation (\ref{DLDE}) 
coincides with the  parametric Lam\'e equation  corresponding to the electrostatic problem  of Theorem \ref{ThR} with fixed 
charges determined by the partial fraction decomposition 
of $\tilde{B}(x)/A(x)$ and a polynomial constraint of the form $R(X)=0$. 
\end{theorem}
 
 \medskip

Theorem \ref{ThR}, and especially Theorem \ref{th:Deg}, together with the classical relation between nondegerate Lam\'e equations and electrostatics, reveal 
a rather general phenomenon. Namely, every Lam\'e differential equation (\ref{LameeqN}) 
where  $A(x)$ has  distinct complex zeros is related to an electrostatic problem no matter what the degree of the polynomial $B(x)$ is, provided only that the Fuch index is nonnegative.
Indeed,  dividing $B(x)$ by $A(x)$ we obtain $B (x)=A(x) r^\prime(x) + \tilde B(x)$. In the classical non-degenerate case when $deg\, B < deg\, A$,  we get $r^\prime(x) =0$ and the partial fraction decomposition of $B(x)/A(x)$ determines the positions and the strengths of all fixed charges. In the degenerate case, the $A$-adjusted   function $r(x)$, satisfying the assumptions of Theorem~\ref{th:Deg}, is simply a primitive of the quotient $r^\prime(x)$. In other words, $r(x)$ is a unique, up to an additive constant, polynomial  with the above properties.  

However, in some situations $r(x)$ is a rational function and not just a polynomial, see Section \ref{S3.3} below. The value $\rho^\ast$ of the constant $\rho$ is uniquely determined by the fact that the zeros $x_1^\ast,\ldots, x_n^\ast$ of a polynomial solution $y(x)$ must satisfy the constraint   
$\sum_{k=1}^n r(x_k^\ast) + \rho^\ast = 0$. Usually $\rho^\ast$ is easily obtained by comparing the coefficients of certain powers of $x$ in the corresponding Lam\'e equation. 
The partial fraction decomposition of  $\tilde B(x)/A(X)$, where $\tilde B(x)$ is the remainder in the above presentation of $B(x)$, determines the fixed charges, if any.

\medskip\noindent
{\it Acknowledgements. } The authors are sincerely grateful to the anonymous referees for their careful reading of the initial version of the manuscript and helpful suggestions. The second  author is sincerely grateful to Universidade Esta\-dual Paulista for the hospitality and excellent research atmosphere during his  visit to S\~{a}o Jos\'{e}  do Rio Preto in July 2017.

\begin{center}
\section{Proofs}\label{S3}
\end{center}

 Consider the  multi-valued analytic function
\begin{equation}\label{extra}
F(X)= \prod_{j=0}^{p} \prod_{i=1}^{n} (x_i-a_j)^{\nu_j} \ \prod_{1\leq i < j \leq n}  (x_j-x_i).
\end{equation}
$F(X)$ is well-defined as a multi-valued function at least in $\bC^n\setminus \mathcal A$ and also on some part of $\mathcal A$ where it vanishes.
 (This function and its generalizations called the {\it master functions} were thoroughly studied  by A.~Varchenko and his coauthors in a large number of publications including  \cite{Va, Va2}.) Although $F(x)$ is multi-valued (unless all $\nu_j$'s are integers), its absolute value  
$$
H(X) := \vert F(X)\vert = \prod_{j=0}^{p} \prod_{i=1}^{n} |x_i-a_j|^{\nu_j} \ \prod_{1\leq i < j \leq n} \vert x_j-x_i\vert.
$$
 is a uni-valued function in $\bC^n\setminus \mathcal A$. Obviously, the energy function \eqref{L} satifies the relation 
$$L(X)=-\log H(x)=-\log \vert F(X) \vert$$
which implies that $L(X)$ is a well-defined  pluriharmonic function in $\bC^n\setminus \mathcal A$, see e.g. \cite{GuRo}. (If all $\nu_j$'s are positive, then $L(x)$ is plurisuperharmonic function in $\bC^n$.) 

\medskip
Although $F(X)$ is multi-valued, its critical points in $\bC^n\setminus \mathcal A$ are given by a well-defined system of algebraic equations and are typically finitely many due to the fact that the ratio of any two branches is constant. (For degenerate cases, these critical points can form subvarieties of positive dimension.)  The following general fact is straightforward.   

\medskip
\begin{lem}\label{VaL2} Let $f_j : \bC^k\to \bC$, $j=1,\dots, N$ be pairwise distinct linear polynomials and let $(y_1,\dots, y_k)$ be coordinates in $\bC^k$.  For every $j$, denote by $H_j\subset \bC^k$ the hyperplane given by $\{f_j=0\}$ and  set 
$T=\bC^k\setminus \bigcup_{j=1}^N H_j.$
Given    a collection of complex numbers $\Lambda=\{\la_j\}_{j=1}^N$, define 
$$\Phi_\Lambda(y_1,\dots, y_k)=\prod_{j=1}^N f_j^{\la_j} .$$
($\Phi_\Lambda$ is a multi-valued holomorphic function defined in $T$.)  
Then the system of  equations defining the critical points of $\Phi_\Lambda$ in $T$ is given by:  
$$\sum_{j=1}^N\la_j\frac{\partial f_j}{\partial y_{\ell}}/f_j=0, \quad \ell=1,\dots, k.   $$
\end{lem} 

(Note that by a critical point of a  holomorphic function we mean a point where its complex gradient vanishes.) Similarly, for any algebraic hypersurface $Y\subset \bC^k$, the critical points of the restriction of $\Phi_\Lambda$ to $Y\cap T$ are well-defined  independently of its branch and can be found by using the complex version of the  method of Lagrange multipliers.  
 
 \begin{lem}\label{Lagr}
 In the notation of Lemma~\ref{VaL2}, let $Y\subset \bC^k$ be an algebraic hypersurface given by $Q(y_1,\dots, y_k)=0$. Then the critical points of the restriction of $\Phi_\Lambda$ to $Y\cap T$ are given by the condition that the gradient of the Lagrange  function $$\mathcal L(y_1,\dots, y_k, \rho)=\Phi_\Lambda - \rho Q(y_1,\dots, y_k)$$ vanishes, where $\rho$ is a complex parameter. The latter condition is given by the system of equations
 \begin{equation}\label{eq:Lagr}
 \sum_{j=1}^N\la_j\frac{\partial f_j}{\partial y_{\ell}}/f_j=\rho \frac{\partial Q}{\partial y_{\ell}}, \quad \ell=1,\dots, k, \quad \text{and} \quad Q(y_1,\dots, y_k)=0.  
 \end{equation}
 (More exactly, if $(y_1^\ast,\dots, y_k^\ast, \rho^\ast)$ solves \eqref{eq:Lagr}, then $(y_1^\ast,\dots, y_k^\ast)$ is a critical point of $\Phi_\Lambda$ restricted to $Y\cap  T$.)
 \end{lem}

\begin{proof} If $\Phi_\Lambda$ were a well-defined holomorphic function defined in $T$ and $Y$ were as above, then the claim of Lemma~\ref{Lagr} has been  obvious because in order  to obtain the critical points of the restriction of a holomorphic function to an algebraic hypersurface one has to find all points on the hypersurface where the complex gradients of the function under consideration and the hypersurface are proportional. 


In our situation, however, $\Phi_\Lambda$ is multi-valued, but the ratio of any two branches is a non-vanishing constant. This implies that at any point of $T$,  the (complex) gradients of any two branches are proportional to each other with a non-vanishing  constant of proportionality independent of the point of consideration.  Therefore, if at some point of $T$ he complex gradient of some branch of $\Phi_\Lambda$ is proportional to that of $Q$,   the same holds at the same point for any other branch of $\Phi_\Lambda$ with (possibly) different constant of proportionality.  The latter observation means that although there are typically infinitely many solutions of \eqref{eq:Lagr} in the variables $(y_1,\dots, y_k, \rho)$, there are only finitely many projections of these solutions to the space $T$, obtained by  forgetting the value of the Lagrange multiplier $\rho$. 
\end{proof}

\begin{remark}{\rm 
Observe that $Y$ has real codimension $2$ and  the equation of proportionality of complex gradients is, in fact, a system of two real equations; the real and the imaginary parts of the proportionality constant can be thought of as two real Lagrange multipliers corresponding to two constraints which express the vanishing of the real and imaginary parts of the polynomial $Q$ defining the hypersurface $Y$. 
 Additionally, the real part and imaginary parts of $\Phi_\Lambda$ are conjugated pluriharmonic functions and therefore have the same set of critical points.} 
 \end{remark}

\medskip
\begin{lem}\label{LF} Under the above assumptions, 
the set of critical points of $L(X)$ in $\bC \setminus \mathcal A$ as well as  the set of critical points of the restriction of $L(X)$ to any $Y\setminus \mathcal A$, where $Y$ is an algebraic hypersurface in $\bC^n$ coincides with those of $F(X)$. 
\end{lem}

Notice that in Lemma~\ref{LF}, the meaning of a \emph{critical point} of the real-valued function $L(X)$ and the meaning of a \emph{critical point} the multi-valued holomorphic  function $F(X)$ are different. In the former case we require that the \emph{real} gradient of $L(X)$ with respect to the real and imaginary parts of the complex coordinates vanishes while in the latter case we require that the complex gradient of $F(X)$ vanishes. Lemma~\ref{LF} holds due to the fact that $L(X)$ is a pluriharmonic function closely related to $F(X)$.

\begin{proof} 
As a warm-up exercise let us prove that the sets of critical points of $L(X)$ and $F(X)$ in $\bC^n\setminus \mathcal A$ coincide. Notice that $F(X)$ is non-vanishing in $\bC^n\setminus \mathcal A$ which means that no branch of $F(X)$ vanishes there. As was mentioned above, 
$$L(X)=-\log \vert F(X)\vert=-\text{Re}\,\left( \Log F(X)\right),$$
where $\Log F(X)$ is a multi-valued holomorphic logarithm function in $\bC^n\setminus \mathcal A$ which is well-defined due to the fact that $F(X)$ is non-vanishing. Similarly to the case of $F(X)$, vanishing of the  complex gradient of $\Log F(X)$ is given by a well-defined system of algebraic equations which coincides with that for $F(X)$.   
  
  Let us now express the real gradient of $L(X)$ using the complex gradient of $\Log F(X)$.  Consider the decomposition of the complex variable $x_k$ into its real and imaginary parts, i.e.   $x_k=u_k+I\cdot v_k$, where $I=\sqrt{-1}$. 

Recall also that due to the Cauchy-Riemann equations, for any (locally) holomorphic function $W(x_1,\dots, x_n)=U(x_1,\dots, x_n)+I\cdot V(x_1,\dots, x_n)$, one has
$$\frac{\partial W}{\partial x_k}=2\frac{\partial U}{\partial x_k}=2I \frac{\partial V}{\partial x_k}, \quad \frac{\partial W}{\partial \bar x_k}=\frac{\partial U}{\partial \bar x_k}=\frac{\partial V}{\partial \bar x_k}=0, \quad k=1,\dots, n.$$
The real gradient $\grad_\bR L$ is given by 
$$\grad_\bR L=\left(\frac{\partial L}{\partial u_1},\frac{\partial L}{\partial v_1},\dots, \frac{\partial L}{\partial u_n},  \frac{\partial L}{\partial v_n}  \right)=-\grad_\bR \text{Re}\,\left( \Log F(X)\right). 
$$
Denoting $G:= - \Log F(X)$ and using the latter relations, we get
\begin{equation}\label{gradR}
\grad_\bR L=\left( -\text{Re}\left( \frac{\partial G}{\partial x_1}\right),\text{Im} \left(\frac{\partial G}{\partial x_1}\right),\dots, -\text{Re}\left(\frac{\partial G}{\partial x_n}\right),\text{Im} \left(\frac{\partial G}{\partial x_n}  \right)     \right), 
\end{equation}
see \cite{KoWi}, p.~3. Observe that $$\text{Re}\left( \frac{\partial G}{\partial x_k}\right)+I \cdot \text{Im} \left(\frac{\partial G}{\partial x_k}\right)=\frac{\partial G}{\partial x_k}=-2  \frac{\partial L}{\partial x_k},$$
implying that
$$\text{Re}\left( \frac{\partial G}{\partial x_k}\right)- I \cdot \text{Im} \left(\frac{\partial G}{\partial x_k}\right)=\overline{\frac{\partial G}{\partial x_k}}=-2  \overline{\frac{\partial L}{\partial x_k}}.$$
Thus $\grad_\bR L$ interpreted as a complex vector in coordinates $(x_1,\dots, x_n)$ is given by 
\begin{equation}\label{grad}
\grad_\bR L=-\left(\overline{\frac{\partial G}{\partial x_1}}, \dots, \overline{\frac{\partial G}{\partial x_n}}\right).
\end{equation}
Therefore, $\grad_\bR L=0$ if and only if $ \grad_\bC \;\Log F(X)=0$, which is equivalent to $\grad_\bC \;F(X)=0$.

\medskip
Let us finally discuss  the situation when one restricts $L(X)$ to an algebraic hypersurface $Y$. If $p\in Y$ is a singular point of $Y$, then both $L(X)$ and $\Log F(X)$ have critical points at $p$. If $p\in Y$ is a nonsingular point, then introduce new complex coordinates $(\tilde x_1,\dots, \tilde x_{n-1}, \tilde x_n)$ adjusted to the tangent plane to $Y$ at $p$. Namely, the origin with respect these coordinates is placed at $p$, the hyperplane spanned by 
 $\left(\tilde x_1,\dots, \tilde x_{n-1}\right)$ coincides with the tangent plane to $Y$ at $p$, and $\tilde x_n$ spans the same line as $\grad_\bC Q(p)$, where $Q$ is the polynomial defining $Y$.  Formulas~\eqref{gradR}  and ~\eqref{grad}   are valid with respect to the new coordinate system as well. The condition that $\left. L(X)\right\vert_Y$ has a critical point at $p$ means that $\grad_\bR L(p)$ lies in the complex line spanned by $\tilde x_n$ which is equivalent to the vanishing of $2n-2$ real quantities 
 $ -\text{Re}\left( \frac{\partial G(0)}{\partial \tilde x_1}\right),\text{Im} \left(\frac{\partial G(0)}{\partial \tilde x_1}\right),\dots, -\text{Re}\left(\frac{\partial G(0)}{\partial \tilde x_{n-1}}\right),\text{Im} \left(\frac{\partial G(0)}{\partial \tilde x_{n-1}}  \right).$
 Analogously, the condition that  $\left. G(X)\right\vert_Y$ has a critical point at $p$ means that $\grad_\bC G(p)$  spans  in the complex line spanned by $\tilde x_n$ which is equivalent to the vanishing of $n-1$ complex quantities 
 $ \frac{\partial G(0)}{\partial \tilde x_1}, \dots, \frac{\partial G(0)}{\partial \tilde x_{n-1}}.$ But vanishing of the former $(2n-2)$ real quantities is obviously equivalent to the vanishing of the latter $(n-1)$ complex quantities.
\end{proof} 

\begin{proof}[Proof of Theorem~\ref{ThR}] We want to find all the critical points of  $L(X)$ subject to the restriction  $R(X)=0$, that is for $X$ lying in  $\Omega\setminus \mathcal A$. (Recall that $L(X)$ is defined in $\bC^n\setminus \mathcal A$ and $\Omega$ is the hypersurface in $\bC^n$ given by $R(X)=0$.) Using Lemma~\ref{LF}, we need to find the critical points of the multi-valued holomorphic function $F(X)$ restricted to $\Omega$. System \eqref{eq:Lagr} of Lemma~\ref{Lagr} provides the corresponding  equations defining the critical points. In the particular case of $F(X)$ and $R(X)$ as above, this system reads as follows
\begin{equation}\label{eq:final}
\sum_{j=0}^p\frac{\nu_j}{x_k-a_j}+\sum_{i\neq k}\frac{1}{x_k-x_i}=\rho r^\prime(x_k),\; k=1,\dots, n\ \text{and} \ R(X)=r(x_1)+\dots + r(x_n)=0.
\end{equation}
This system contains $n+1$ equations in the $(n+1)$ variables $x_1,\dots, x_n$ and $\rho$.

Abusing our notation, assume for the moment that $X=(x_1,\dots, x_n)$ is not the set of variables for $\bC^n$, but some concrete complex vector in $\bC^n\setminus \mathcal A$ solving the system~\eqref{eq:final}. Introducing the polynomials $y(x)=\prod_{j =1}^n (x-x_j)$, $y_k(x)=y(x)/(x-x_k)$ and taking into account the obvious relations  
$y_k(x_k) = y^\prime (x_k)$,  $2 y_k^\prime(x_k) = y^{\prime\prime} (x_k)$, we conclude  that the first $n$ equations of the system  \eqref{eq:final} are equivalent to the system given by  
\begin{equation}\label{eq:int}
A(x_k) y^{\prime\prime} (x_k) + \left ( 2B(x_k) - \rho D(x_k) \right ) y^\prime(x_k) = 0,\ \ k=1,\ldots, n.
\end{equation}
Indeed, multiplying the $k$-th equation of \eqref{eq:final}  by $A(x_k)y_k(x_k)$, we obtain exactly the $k$-th equation of \eqref{eq:int}. 

Under our assumptions, 
$$
D(x_k) := A(x_k) \frac{\partial R(X)}{\partial x_k}=A(x_k)r^\prime(x_k)
$$
is  a polynomial in $x_k$ of degree $q$. Since the first term in \eqref{eq:int} is a polynomial of degree $n+p-1$ and the second one  
is a polynomial of degree $n+q-1$, then by the fundamental theorem of algebra, there exists a polynomial $V(x)$, of degree $\max \{p-1,q-1\}$, such that 
$$
A(x) y^{\prime\prime} (x) + \left ( 2B(x) - \rho D(x) \right ) y^\prime(x)+V (x)y = 0,
$$ 
where $y(x)=\prod_{j =1}^n (x-x_j)$. 
\end{proof}

Notice that the condition that  $R(X)$ is a rational symmetric function of a special form was imposed to guarantee  that the parametric Lam\'e equation 
admits a polynomial solution.

\begin{proof}[Proof of Theorem~\ref{th:Deg}] 
Indeed, observe that the above conditions show that the degenerate Lam\'e equation (\ref{DLDE}) takes the form 
$$
A(x) y^{\prime\prime} + 2 (\tilde B(x)  + A(x) r^\prime(x))y^{\prime}  + V(x) y
= 0.
$$
According to the second statement of Theorem \ref{ThR}, the existence of a pair $(V(x),y(x))$ yields that $X^\ast=(x_1^\ast, \dots, x_n^\ast)$ is a critical point of the restriction of the energy function $L(X)$ to the hypersurface given by $R(X)=0$.   Here $V(x)$ is a Van Vleck polynomial 
satisfying  $\deg V \leq \max\{ \deg A -2, \deg q-1\}$ and $y(x)=(x-x_1^\ast)\cdots(x-x_n^\ast)$  is  a Stieltjes polynomial which satisfies the 
 latter differential equation where $X^\ast=(x_1^\ast, \dots, x_n^\ast)$ obeys the restriction  $R(X^\ast)=0$.  
\end{proof}

 \section{Examples} 
  
\subsection{Hermite polynomials in disguise} A straightforward change of variables in the differential equation (\ref{HDE}) satisfied by the  Hermite polynomials $H_n(x)$ implies that for any $m\in \mathbb{N},$
the polynomial $y(x) =y_{m,n}(x):= H_n(x^m)$ solves the degenerate Lam\'e differential equation 
$$
x\, y^{\prime\prime}(x) - (2\, m\, x^{2m} + m-1)\ y^{\prime}(x) + 2\, m^2 n\, x^{2m-1} y(x) = 0.
$$
For any fixed $m$, the zeros of  $H_n(x^m)$ are located at the intersections of the $2\, m$ rays emanating  from the origin with the slopes $e^{i\pi j/m}$, $j=0,\ldots, 2m-1$ and $[n/2]$ circles with the radii 
$(h_k)^{1/m}$, $k=1,\ldots, [n/2]$, where  $h_k$ are the positive zeros  of $H_n(x)$.  When $n$ is  odd, there is an additional zero of multiplicity $m$ at the origin.     

  Theorem \ref{th:Deg} implies that the coordinates of these $mn$ zeros  form  a critical point  of the logarithmic energy of the  electrostatic field generated by the moving charges together with 
the negative charge $-(m-1)/2$ at the origin, where the $mn$ movable charges $x_k$ are subject to the constraint 
$$
\sum_{k=1}^{mn} x_k^{2m} = \frac{m n (n-1)}{2}.
$$

\subsection{Laguerre polynomials in disguise}

A  procedure similar to that in the previous example shows that for any  $m\in \mathbb{N}$,
the polynomial $y(x)=y_{m,n}(x): = L_n^{\alpha}(x^m)$ solves the differential equation 
$$
x\, y^{\prime\prime}(x) + (1+\alpha\, m - m\, x^{m})\ y^{\prime}(x) + m^2 n\, x^{m-1} y(x) = 0.
$$
For any fixed $m$, the zeros of  $L_n^{\alpha}(x^m)$ are located at the intersections of the $m$ rays emanating from to origin with the slopes  $e^{2\pi i/j}$, $j=0,\ldots, m-1$ and the $n$ circles with the radii 
$(\ell_k)^{1/m}$, $k=1,\ldots, n$, where $\ell_k$ are the zeros  of $L_n^{\alpha}(x)$. 
    
 Theorem \ref{th:Deg} implies that the coordinates of these  $mn$ zeros  form a critical point of the logarithmic energy of the electrostatic field generated by the moving charges together with 
the charge $(1+\alpha m)/2$ at the origin, where the $mn$ movable charges $x_k$ are subject to the constraint 
$$
\sum_{k=1}^{mn} x_k^m = m n (n+\alpha).
$$

\subsection{Laguerre polynomials and electrostatic problem with a rational constraint}
\label{S3.3}

Substituting  $x \mapsto x+1/x$ in $L_n^{\alpha}(x)$,  we conclude that the polynomial of degree $2n$  $$Y(x) = x^n L_n^{\alpha}(x+1/x)$$  solves the differential equation
$$
A(x)\, y^{\prime\prime}(x) + {B}(x)\ y^{\prime}(x) + V(x) y(x) = 0,
$$
with
\begin{eqnarray*}
A(x) & = & x^6 - x^2,\\
{B}(x) & = & - x^6 + x^2 + (a+1-2n )x^5 + x^4 - 2(a+2) x^3 + (a+ 2 n -1) x -1,\\
V(x) & = & 2 n x^5 +n(n-a) x^4 - 4 n x^3 + 2 n (a+2) x^2 + 2 n x -n(n+a).
\end{eqnarray*} 

Now take $r(x)=x+1/x$. Then $A(x) r^\prime(x)= (x^2+1) (x^2-1)^2$ and 
$
{B}(x) = - A(x) r^\prime(x) + 2 \tilde B(x), 
$
where 
$2 \tilde B(x) = (a+1-2n )x^5 - 2(a+2) x^3 + (a+ 2 n -1) x$. 

\medskip The partial fraction decomposition of $\tilde B(x)/A(x)$ is given by
$$
\frac{\tilde B(x)}{A(x)} = \left(-n+\frac{1 - a}{2}  \right) \frac{1}{x} - \frac{1}{2} \left( \frac{1}{x - 1} + \frac{1}{x +1} \right) + \frac{a+1}{2} \left( \frac{1}{x - i} + \frac{1}{x + i} \right).
$$ 
  Theorems 1 and 2 imply that the $2n$ zeros of $x^n L_n^{\alpha}(x+1/x)$ are the coordinates of a critical point of the logarithmic energy of the electrostatic field generated by the movable charges together with  
the following five fixed charges. One charge equal to $-n+ (1 - a)/2$ is placed at the origin; two charges equal to $-1/2$ are placed  at $\pm 1$,  and two charges $(a+1)/2$ are placed at $\pm i$. 

\medskip
The $2n$ unit movable charges obey the constraint  
$$
\sum_{k=1}^n (x_k + 1/x_k) = \frac{n\, ( (\alpha+1)(n+\alpha)+1)}{\alpha+1}. 
$$
Since the relation $x_k + 1/x_k=\ell_k$  associates each zero $\ell_k$ of $L_n^{\alpha}(x)$ to 
two zeros of $x^n L_n^{\alpha}(x+1/x)$, we conclude that  the critical points are either positive reals or belong to the semicircle $\{ x\in \mathbb{C} : \ |x|=1, \Re(x)>0 \}$.

 \subsection{Schr\"odinger-type equations}  Now we consider another type of analogs of the equation satisfied by the Hermite polynomials. We say that \eqref{Lame} is of \emph{Schr\"odinger-type} if 
 $A(x)$ is a non-vanishing  constant (which we can always assume equals to $1$).  In this case, there are no fixed charges and the system of  equations defining the equilibrium (usually called the Bethe ansatz) is given by 
 \begin{equation}\label{Bethe}
 \sum_{j\neq k}\frac{1}{x_k-x_j}=-B(x_k),\; k=1,\dots, n.
 \end{equation}
 
Denoting by $r(x)$  a primitive function of $B(x)$, we observe that equation \eqref{Bethe} determines critical points of the Vandermonde function
 $$Vd(x_1,\dots, x_n):=\prod_{1\le i<j\le n}|x_i-x_j|$$
 on the hypersurface $\mathcal H$ given by the equation 
 $$
 R(x_1,\dots, x_n):=r(x_1)+r(x_2)+\dots + r(x_n)=C
 $$ 
for an appropriate constant $C$. (Notice that  the critical points of the Vandemonde function and its generalizations to several types of hypersurfaces have been studied in  \cite{LOS, LOS2, Lu}.)

\medskip
Some interesting examples of Schr\"odinger-type equations  are satisfied by the Laguerre polynomials of certain degrees with special values of  parameter $\alpha$.
More precisely, for a given $m\in \mathbb{N}$,  consider  the equation 
\begin{equation}
\label{SLDE}
y^{\prime\prime}(x) - (m+1)\, x^m y^{\prime}(x) + m(m+1)\, n\, x^{m-1}  y(x) =0.
\end{equation}
One expects that  eventual polynomial solutions of \eqref{SLDE} must have degree $mn$. Observe that a basis of 
 linearly  independent solutions of \eqref{SLDE}  is given by 
$$
    \, _1F_1(\, -mn/(m+1), \,1- 1/(m+1),\, x^{m+1})
$$
and
$$
x \cdot \, _1F_1(\, (1-mn)/(m+1),\, 1+ 1/(m+1),\, x^{m+1})  ,
$$
where $\, _1F_1(a, b, x)$ is the basic hypergeometric function given by 
$$
\, _1F_1(a, b, x) := \sum_{j=0}^\infty \frac{(a)_j}{(b)_j} \frac{x^j}{j!}.
$$
Here $(t)_j$ is the standard Pochhammer symbol defined by $(t)_0:=1$ and  $(t)_j:=t(t+1)\cdots (t+j-1)$, $j$ being a positive integer. 

\medskip
It is clear that $\, _1F_1(a, b, x) $ reduces to a polynomial if and only if $a$ is a nonpositive integer, i.e., $a=-N$ with $N\in \mathbb{N}$. 
Moreover for $b>0$, these polynomials coincide with the 
Laguerre polynomials since  $L_N^{\alpha}(x)$ is a constant multiple of $\, _1F_1(-N, \alpha+1, x)$. Therefore,  
 equation (\ref{SLDE}) has a polynomial solution if and only if either $(1-mn)/(m+1)$ or $-mn/(m+1)$ is a negative integer 
$-N$. Moreover,  in these cases   $x L_N^{1/(m+1)}(x^{m+1})$ and  $L_N^{1/(m+1)}(x^{m+1})$ are respective polynomial solutions.

\medskip
 Let us first  consider the case when  
$(1-mn)/(m+1)=-N$. It is not difficult to observe that for a fixed 
$m\in \mathbb{N},$ the pairs $(n,N)=(d m+m-1,  dm-1)$ satisfy the above relation for any nonnegative integer $d$. Therefore for every $d\in \mathbb{N}$,  the polynomials 
$x L_{dm-1}^{1/(m+1)}(x^{m+1})$ are solutions of (\ref{SLDE}). Similar reasoning yields that the relation 
$mn/(m+1)=N$, with fixed $m$ is satisfied by the pairs $(n,N)=((m+1)d, md)$, implying that  $L_{dm}^{1/(m+1)}(x^{m+1})$  are also solutions 
of (\ref{SLDE}) for every $d\in \mathbb{N}$.

\medskip
Applying Theorem \ref{th:Deg}  to (\ref{SLDE}) we conclude the following. Since $A(x)\equiv 1$, there are 
no fixed charges. There are $Nm$ movable charges $x_k$ in the complex plane  which obey the constraint 
$$
\sum_{k=1}^{mn} x_k^{m+1} = C,
$$
where $C=(m+1)N(N+1)/(m+1))$. In one of the situations $N=md-1$, or equivalently $N=(mn-1)/(m+1)$ and in the other case $N=md-1=mn/(m+1)$. In the first case,  
one charge is at the origin and the remaining  $nm-1=N(m+1)=(m+1)(md-1)$ are placed on the $m+1$ rays $e^{2j \pi i/(m+1)}$, $j=0,\ldots,m$. In the second case,  there are only $nm=N(m+1)=dm(m+1)$ 
charges placed on the latter  $m+1$ rays.

 \subsection{Relativistic Hermite polynomials}
 
\medskip
 Our last example illustrates how polynomial solutions of a non-degenerate Lam\'e equation depending on a parameter  become  polynomial solutions of a degenerate Lam\'e equation in case  of the zeros of the so-called relativistic Hermite polynomials $H_n^N(x) $, see \cite{GavVA}. Namely,   for  
 any positive number $N>0$,  define  $H_n^N(x) $ by the Rodrigues formula
$$
H_n^N(x) := (-1)^n \left( 1 + \frac{x^2}{N} \right)^{N+n} \left( \frac{d}{dx} \right)^n \frac{1}{(1+x^2/N)^N}.
$$
One can show that  for $N> 1/2$,  polynomials $H_n^N(x)$ are orthogonal with respect to the following varying weight:  
$$
\int_{-\infty}^{\infty} H_m^N(x) H_n^N(x)\, \frac{dx} {(1+x^2/N)^{N+1+(m+n)/2}} = c_n \delta_{mn}.
$$ 
\medskip
Additionally,  $H_n^N(x)$ is the unique polynomial solution of the non-degenerate Lam\'e  equation
$$
( N + x^2)\, y^{\prime\prime} - 2\, (N+n-1)\, x\, y^{\prime} + n\, (2N+n-1)\, y = 0.
$$
\medskip
The partial fraction decomposition of $B(z)/A(z)$ is given  by
$$
\frac{B(x)}{A(x)} = - \frac{N+n-1}{2} \left( \frac{1}{x-i\sqrt{N}} + \frac{1}{x+i\sqrt{N}} \right).
$$ 

\medskip
Placing two fixed equal negative charges  $-(N+n-1)/2$  at $\pm i\sqrt{N}$ 
 and $n$ unit movable charges on the real line,  we obtain after a straightforward calculation, that the zeros of $H_n^N(x)$ coincide with the unique equilibrium configuration  of $n$ movable  unit 
 charges where the energy attains its minimum. In fact, this minimum is  global.  
 
\medskip 
Let us briefly discuss  how this equilibrium configuration depends on the positive parameter $N$. One can deduce the following.  
\begin{itemize}
\item When $N$ is a small positive number, then the fixed negative charges are located close to the origin and their total strength equals $-(n-1)$. Therefore,  they attract all movable 
charges to the origin. 
\item When $N$ grows, then  all movable charges move away from the origin because the force of attraction of the fixed negative charges decreases. This  can be proved rigorously via a refinement of Sturm's comparison theorem obtained in \cite{DKD}.

\item When $N\rightarrow +\infty$, the force of attraction of the fixed negative charges decreases 
because they move away from the real line, but at the same time their  strength increases in such a way that the location of each movable charge has a limit, see Fig.~1.
\end{itemize}

\begin{figure}\label{fig1}
      \includegraphics[width=7.0cm]{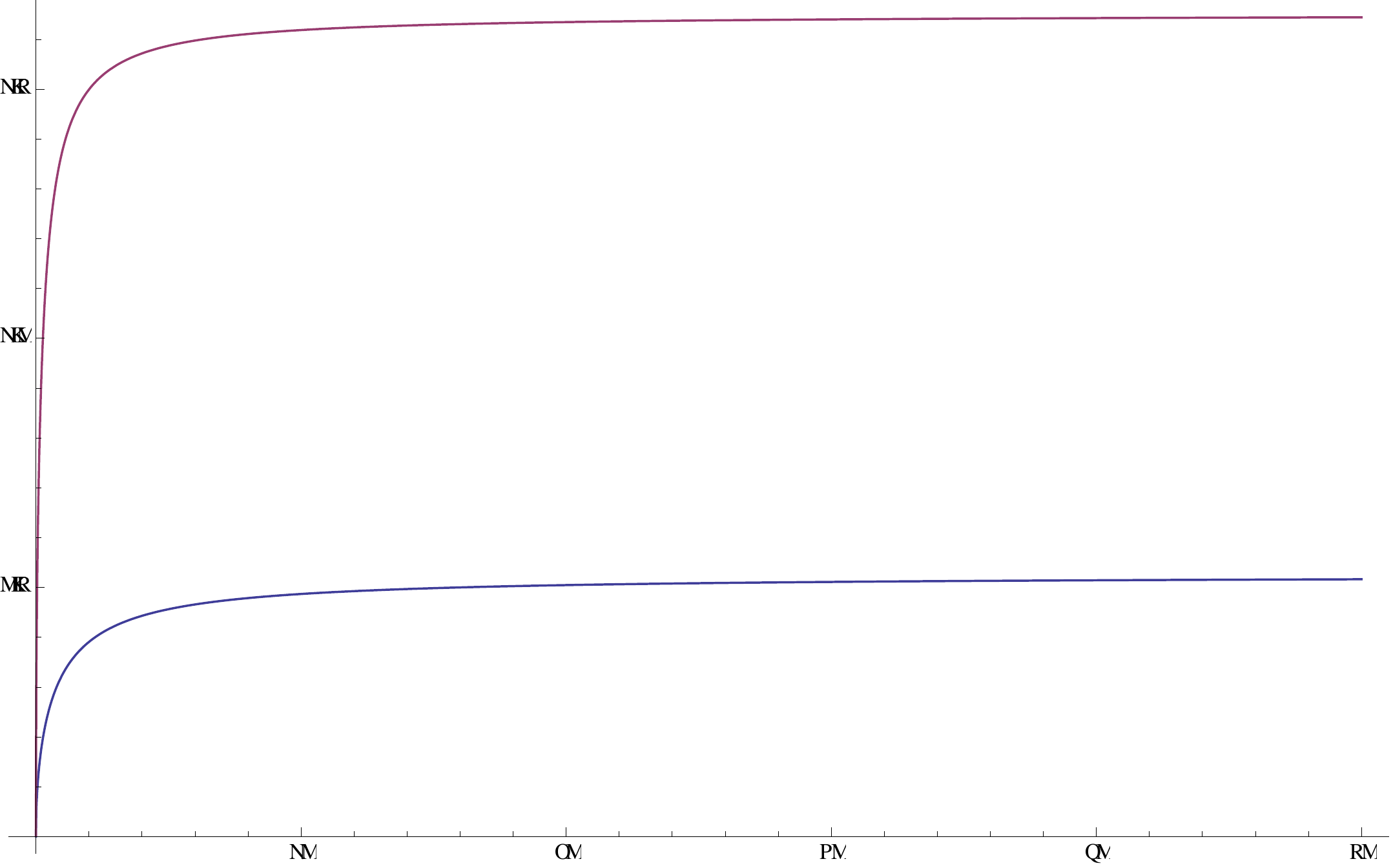}\\
\caption{Dependence of the positions of the two positive zeros of $H_4^N(x)$ on  $N$. (The remaining two zeros of $H_4^N(x)$ are negative.)}
     \end{figure}\vspace{0cm}

\medskip
The conclusion based on the above observations is as follows.  
Since $H_n^N(x)$ converges locally uniformly to the Hermite polynomial $H_n(x)$ as $N$ goes to infinity, the zeros of $H_n^N(x)$ converge to those of $H_n(x)$.
 Therefore, when $N$ increases, the negative charges at $\pm i \sqrt{N}$ increase in absolute value. The corresponding  equilibrium configuration formed by the coordinates of the zeros of $H_n^N(x)$ is such that the latter zeros move monotonically to those of $H_n(x)$.   What is special about this phenomenon is that the influence of the increasing negative charges   at 
 $\pm i \sqrt{N}$ do not ``disappear at infinity" as one can suspect, but it  transforms into a constraint.

\end{document}